\documentclass[12pt,oneside,english]{amsart}
\usepackage[latin1]{inputenc}
\usepackage{amssymb}

\theoremstyle{plain}
\newtheorem{theorem}{Theorem}

\newtheorem{proposition}[theorem]{Proposition}

\theoremstyle{definition}
\newtheorem{definition}[theorem]{Definition}
\newtheorem{example}[theorem]{Example}

\theoremstyle{remark}
\newtheorem{remark}[theorem]{Remark}

\usepackage{babel}

\makeatother
\begin{document}
\baselineskip=17pt

\title[On intervals $(kn,(k+1)n)$ containing a prime for all $n>1$]
{On intervals $(kn,(k+1)n)$ containing a prime for all $n>1$}
\author{Vladimir Shevelev}
\address{Department of Mathematics \\Ben-Gurion University of the
 Negev\\Beer-Sheva 84105, Israel; e-mail: shevelev@bgu.ac.il}
 \author{Charles R. Greathouse IV}
 \address{United States; e-mail: charles.greathouse@case.edu}
 \author{Peter J. C. Moses}
 \address{United Kingdom; e-mail: mows@mopar.freeserve.co.uk}
 \subjclass{MSC 2010: 11A41. Key words and phrases: prime numbers,
 generalized Ramanujan primes}
 \begin{abstract}
 We study values of $k$ for which the interval  $(kn,(k+1)n)$ contains a prime for
every $n>1.$ We prove that the list of such integers $k$ includes
 $k=1,2,3,5,9,14, $ and no others, at least for $k\leq 50,000,000.$ For every known
 $k$ of this list, we give a good upper estimate of the smallest $N_k(m),$ such
 that, if $n\geq N_k(m),$ then the interval $(kn,(k+1)n)$ contains at least $m$ primes.
\end{abstract}
\maketitle
\section{Introduction and main results}
In 1850, P. L. Chebyshev proved the famous Bertrand postulate (1845) that
every interval $[n,2n]$ contains a prime (for a very elegant version of his proof,
 see Theorem 9.2 in \cite{10}). Other nice proofs were given by
S. Ramamujan in 1919 \cite{8} and P. Erd\H{o}s in 1932 (reproduced in \cite{4}, pp.171-173). In 2006, M. El. Bachraoui \cite{1} proved
that every interval $[2n,3n]$ contains a prime, while A. Loo \cite{6} proved the same
statement for every interval $[3n,4n].$  Moreover, A. Loo found a lower estimate
for the number of primes in the interval $[3n,4n].$ Note also that already in 1952
 J. Nagura \cite{7} proved that, for $n\geq25,$ there is always a prime between $n$ and
$\frac{6}{5}n.$ From his result it follows that the interval $[5n,6n]$
always contains a prime.
In this paper we prove the following.
\begin{theorem}\label{t1}
The list of integers \;$k$ for which every interval $(kn, (k+1)n),\; n>1,$ contains
 a prime includes\; $k=1,2,3,5,9,14$ and no others, at least for $k\leq 50,000,000.$
 \end{theorem}
 Besides, in this paper, for every $k=1,2,3,5,9,14,$ we give an algorithm for
  finding the smallest $N_k(m),$ such that, for $n\geq N_k(m),$ the interval
 $(kn,(k+1)n)$ contains at least $m$ primes.
\section{Case $k=1$}
Ramanujan \cite{8} not only proved Bertrand's postulate but also indicated the
smallest integers  $\{R(m)\},$ such that, if $x\geq R(m),$ then the
interval $\left(\frac{x}{2}, x\right]$ contains at least $m$ primes, or, the same,
 $\pi(x)-\pi(x/2)\geq m.$ It is easy to see that here it is sufficient to consider
 \emph{integer} $x$ and it is evident that every term of $\{R(m)\}$ is prime.
  The numbers $\{R(m)\}$ are called \emph{Ramanujan primes} \cite{14}.
  It is the sequence (A104272 in \cite{13}):
 \begin{equation}\label{1}
 2, 11, 17, 29, 41, 47, 59, 67, 71, 97,...
 \end{equation}
Since $\pi(x)-\pi(x/2)$ is not a monotonic function,
to calculate the Ramanujan numbers one should have an effective upper estimate
of $R(m).$  In \cite{8} Ramanujan showed
 that
 \begin{equation}\label{2}
\pi(x)-\pi(x/2)>\frac{1}{\ln x}\left(\frac{x}{6}-3\sqrt{x}\right),\; x>300.
 \end{equation}
 In particular, for $x\geq324,$ the left hand side is positive and thus $\geq1.$
 Using direct descent, he found that $\pi(x)-\pi(x/2)\geq1$ already from
 $x\geq2.$ Thus $R(1)=2$ which proves the Bertrand postulate. Further, e.g.,
  for $x\geq400,$ the left hand side of (\ref{2}) is more than 1 and thus $\geq2.$
  Again, using direct descent, he found that $\pi(x)-\pi(x/2)\geq2$ already from
 $x\geq11.$ Thus $R(2)=11,$ etc.
 Sondow \cite{14} found that $R(m)<4m\ln(4m)$ and conjectured that
\begin{equation}\label{3}
R(m)<p_{3m}
 \end{equation}
which was proved by Laishram \cite{5}. Since, for $n\geq2,$ $p_n\leq en\ln n$
 (cf. \cite{3}, Section 4), then (\ref{3}) yields $R(m)\leq3em\ln(3m),\; m\geq1.$
Set $x=2n.$ Then, if $2n\geq R(m),$ then $\pi(2n)-\pi(n)\geq m.$ Thus the interval
$(n,2n)$ contains at least $m$ primes, if
$$n\geq\left\lceil\frac{R(m)+1}{2}\right\rceil=\begin{cases}
2,\;\;if\;\;m=1,\;\\ \frac{R(m)+1}{2},\;\;if\;\;m\geq2.\end{cases}$$

Denote by $N_1(m)$ the smallest number such that, if $n\geq N_1(m),$ then the
interval $(n,2n)$ contains at least $m$ primes. It is clear, that $N_1(1)=R(1)=2.$
If $m\geq2,$ formally the condition $x=2n\geq 2N_1(m)$ is not stronger than the
 condition $x\geq R(m),$ since the latter holds for $x$ even and odd. Therefore, for
 $m\geq2,$ we have $N_1(m)\leq \frac{R(m)+1}{2}.$ Let us show that in fact we
 have here the equality.
 \begin{proposition}\label{p2}
 For $m\geq2,$
 \begin{equation}\label{4}
 N_1(m)=\frac{R(m)+1}{2}.
 \end{equation}
 \end{proposition}
\begin{proof} Note that the interval $\left(\frac{R(m)-1}{2}, R(m)-1\right)$
  cannot contain more than $m-1$ primes. Indeed, it is an interval of type
  $\left(\frac{x}{2},x\right)$ for integer $x$ and the following such interval is
  $\left(\frac{R(m)}{2}, R(m)\right).$ By the definition, $R(m)$ is the \emph{smallest} number
  such that if $x\geq R(m),$ then $\{(\frac{x}{2}, x)\}$ contains $\geq m$ primes.
 Therefore, the supposition that already interval $\left(\frac{R(m)-1}{2}, R(m)-1\right)$
   contains $\geq m$ primes contradicts the minimality of $R(m).$
 Since the following interval of type $(y,\;2y)$ with integer $y\geq\frac{R(m)-1}{2}$ is
    $\left(\frac{R(m)+1}{2},\; R(m)+1\right),$ then (\ref{4}) follows.
  \end{proof}

So the sequence $\{N_1(m)\},$ by (\ref{1}), is (A084140
 in \cite{13})
 \begin{equation}\label{5}
 2, 6, 9, 15, 21, 24, 30, 34, 36, 49, ...
 \end{equation}
 \section{Generalized Ramanujan numbers}
Further our research is based on a generalization of Ramanujan's method. With this
 aim, we define generalized Ramanujan numbers (cf. \cite{12}, Section 10, and
  earlier (2009) comment in A164952 \cite{13}).
 \begin{definition}\label{d3}
 Let $v>1$ be a real number. A $v$-Ramanujan number  $(R_v(m)),$ is the smallest
 integer such that if $x\geq R_v(m),$ then $\pi(x)-\pi(x/v)\geq m.$
 \end{definition}
 It is known \cite{10} that all $v$-Ramanujan numbers are primes.
 In particular, $R_2(m)=R(m),\; m=1,2,...,$ are the proper Ramanujan primes.
 \begin{definition}\label{d4}
 For a real number $v>1$ the \emph{$v$-Chebyshev number} $C_v(m)$  is the smallest integer,
  such that if $x\geq C_v(m),$ then $\vartheta(x)-\vartheta(x/v)\geq m\ln x,$
  where $\vartheta(x)=\sum_{p\leq x}\ln p$ is the Chebyshev function.
 \end{definition}
 Since  $\frac{\vartheta(x)-\vartheta(x/v)}{\ln x}$ can enlarge on 1 only when $x$
  is prime, then all $v$-Chebyshev numbers $C_v(m)$ are primes.
\begin{proposition}\label{p5}
  We have
 \begin{equation}\label{6}
 R_v(m)\leq C_v(m).
 \end{equation}
 \end{proposition}
 \begin{proof} Let $x\geq C_v(m).$ Then we have
 \begin{equation}\label{7}
 m\leq \frac {\vartheta(x)-\vartheta(x/v)}{\ln x}=\sum_{\frac{x}{v}
 <p\leq x}\frac{\ln p}{\ln x}\leq \sum_{\frac{x}{v}<p\leq x}1=
 \pi(x)-\pi(x/v).
 \end{equation}
 Thus, if $x\geq C_v(m),$ then \emph{always} $\pi(x)-\pi(x/v)\geq m.$
 By the Definition \ref{d3}, this means that  $R_v(m)\leq C_v(m).$
 \end{proof}
 Now we give an upper estimates for $C_v(m)$ and $R_v(m).$
  \begin{proposition}\label{p6}
  Let $x=x_v(m)\geq2$ be any number for which
   \begin{equation}\label{8}
\frac{x}{\ln x}\left(1-\frac{1300}{\ln^4x}\right)\geq \frac {vm}{v-1}.
 \end{equation}
 Then
\begin{equation}\label{9}
  R_v(m)\leq C_v(m)\leq x_v(m).
  \end{equation}
  \end{proposition}
  \begin{proof}
We use the following inequality of Dusart \cite{3} (see his Theorem 5.2):
$$|\vartheta(x)-x|\leq\frac{1300x}{\ln^4x},\; x\geq2.$$
  Thus we have
 $$\vartheta(x)-\vartheta(x/v)\geq x\left(1-\frac{1}{v}-1300\left(\frac{1}{\ln^4x}-
 \frac{1}{v\ln^4\frac{x}{v}}\right)\right)$$
 $$\geq x\left(1-\frac{1}{v}\right)\left(1-\frac{1300}{\ln^4x}\right). $$
If now
 $$x\left(1-\frac{1}{v}\right)\left(1-\frac{1300}{\ln^4x}\right)\geq m\ln x,\; x\geq x_v(m),$$
 then
 $$\vartheta(x)-\vartheta(x/v)\geq m\ln x,\;x\geq x_v(m)$$
 and, by the Definition \ref{d4}, \;$C_v(m)\leq x_v(m).$  So, according to
 (\ref{6}), we conclude that $R_v(m)\leq x_v(m).$
 \end{proof}
 \begin{remark}\label{r7}
  In fact, in Theorem 5.2 \cite{3} Dusart gives several inequalities of the form
  $$|\vartheta(x)-x|\leq\frac{ax}{\ln^bx},\; x\geq x_0(a,b).$$
  In the proof we used the maximal value $b=4.$
  However, with the computer point of view, the values $a=1300, \;b=4$ from
   Dusart's theorem not always are the best.
   The analysis for $x\geq25$ shows that the condition 
  $$x(1-\frac{1}{v})\left(1-\frac{ax}{\ln^bx}\right)\geq m\ln x$$
is the weakest and thus satisfies for the smallest $x_v(x)=x_v(a,b),$ if to use the following
values of $a$ and $b$ from Dusart's theorem: \newline
  $a=3.965,\;b=2$ for $x$ in range $(25,\;7\cdot10^7];$\newline
  $a=\;1300,\;b=4$ for $x$ in range $(7\cdot10^7,\;10^9];$\newline
  $a=0.001,\;b=1$ for $x$ in range $(10^9,\;8\cdot10^9];$\newline
  $a=\;\;0.78,\;b=3$ for $x$ in range $(8\cdot10^9,\;7\cdot 10^{33}];$\newline
  $a=\;1300,\;b=4$ for $x>7\cdot 10^{33}.$
 \end{remark}
 Proposition \ref{p6} gives the terms of sequences $\{C_v(m)\},\;\{R_v(m)\}$
  for every $v>1, \;m\geq1.$ In particular, if $k=1$ we find  $\{C_2(m)\}:$
  $$11,17,29,41,47,59,67,71,97,101,107,127,149,151,167,179,223,$$
  $$229,233,239,241,263,269,281,307,311,347,349,367,373,401,409,$$
 \begin{equation}\label{10}
 419,431,433,443,...\;.
  \end{equation}
  This sequence requires a separate comment. We observe that up to $C_2(100)=1489$
  only two terms of this sequence $(C_2(17)=223$ and $C_2(36)=443)$ are
  not Ramanujan numbers, and the sequence is missing only the following Ramanujan
  numbers: 181,227,439,491,1283,1301 and no others up to 1489. The latter observation
  shows how much the ratio $\frac{\vartheta(x)}{\ln x} $ exactly approximates
   $\pi(x).$\newline
  \indent Further, for $v=\frac{k+1}{k},$ we find  the following sequences:\newline
  for $k=2,\;\{C_v(m)\},$
  \begin{equation}\label{11}
 13, 37, 41, 67, 73, 97, 127, 137, 173, 179, 181, 211, 229, 239,...\;;
  \end{equation}
  for $k=2,\;\{R_v(m)\},$
 \begin{equation}\label{12}
 2, 13, 37, 41, 67, 73, 97, 127, 137, 173, 179, 181, 211, 229, 239, ...\;;
  \end{equation}
  for $k=3,\;\{C_v(m)\},$
  \begin{equation}\label{13}
  29, 59, 67, 101, 149, 157, 163, 191, 227, 269, 271, 307, 379,...\;;
  \end{equation}
  for $k=3,\;\{R_v(m)\},$
 \begin{equation}\label{14}
  11, 29, 59, 67, 101, 149, 157, 163, 191, 227, 269, 271, 307, 379,...\;;
  \end{equation}
   for $k=5,\;\{C_v(m)\},$
  \begin{equation}\label{15}
 59, 137, 139, 149, 223, 241, 347, 353, 383, 389, 563, 569, 593,...\;;
  \end{equation}
  for $k=5,\;\{R_v(m)\},$
 \begin{equation}\label{16}
   29, 59, 137, 139, 149, 223, 241, 347, 353, 383, 389, 563, 569, 593,...\;;
  \end{equation}
 for $k=9,\;\{C_v(m)\}$,
  \begin{equation}\label{17}
 223, 227, 269, 349, 359, 569, 587, 593, 739, 809, 857, 991, 1009,...\;;
  \end{equation}
  for $k=9,\;\{R_v(m)\},$
 \begin{equation}\label{18}
  127, 223, 227, 269, 349, 359, 569, 587, 593, 739, 809, 857, 991, 1009,...\;;
  \end{equation}
  for $k=14,\;\{C_v(m)\},$
  \begin{equation}\label{19}
  307, 347, 563, 569, 733, 821, 1427, 1429, 1433, 1439, 1447, 1481,...\;;
  \end{equation}
  for $k=14,\;\{R_v(m)\},$
 \begin{equation}\label{20}
  127, 307, 347, 563, 569, 733, 1423, 1427, 1429, 1433, 1439, 1447,...\;.
  \end{equation}
  \section{Estimates of type (\ref{3})}
  \begin{proposition}\label{p8}
  We have
    \begin{equation}\label{21}
 C_2(m-1)\leq p_{3m},\;m\geq2;
  \end{equation}
   \begin{equation}\label{22}
 R_{\frac{3}{2}}(m)\leq p_{4m},\; m\geq1;\; C_{\frac{3}{2}}(m-1)\leq p_{4m},\;
  m\geq2;
  \end{equation}
    \begin{equation}\label{23}
 R_{\frac{4}{3}}(m)\leq p_{6m},\; m\geq1;\;C_{\frac{4}{3}}(m-1)\leq p_{6m},\;
 m\geq2;
  \end{equation}
    \begin{equation}\label{24}
 R_{\frac{6}{5}}(m)\leq p_{11m},\; m\geq1;\;C_{\frac{6}{5}}(m-1)\leq p_{11m},\;
  m\geq2;
  \end{equation}
    \begin{equation}\label{25}
 R_{\frac{10}{9}}(m)\leq p_{31m},\; m\geq1;\;C_{\frac{10}{9}}(m-1)\leq p_{31m},\;
  m\geq2;
  \end{equation}
    \begin{equation}\label{26}
 R_{\frac{15}{14}}(m)\leq p_{32m},\; m\geq1;\;C_{\frac{15}{14}}(m-1)\leq p_{32m},\;
  m\geq2.
  \end{equation}
  \end{proposition}
  \begin{proof}
  Firstly let us find some values of $m_0=m_0(k),$ such that, at least, for
   $m\geq m_0$ all formulas (\ref{21})-(\ref{26}) hold.
  According to (\ref{8})-(\ref{9}), it is sufficient to show that, for $m\geq m_0,$
   we can take $p_{tm},$ where $t=3,4,6,11,31,32$ for formulas (\ref{21})-(\ref{26})
   respectively, in the capacity of $x_v(m).$ As we noted in Remark \ref{r7},
    in order to get possibly smaller values of $m_0,$ we use, instead of (\ref{8}),
    the estimate
    \begin{equation}\label{27}
\frac{x}{\ln x}\left(1-\frac{3.965}{\ln^2x}\right)\geq \frac {vm}{v-1}.
 \end{equation}
 In order to get $x=p_{mt}$ satisfying this inequality, note that \cite{11}
 $$ p_n\geq n\ln n.$$
 Therefore, it is sufficient to consider $p_{mt}$ satisfying the inequality
  $$\ln p_{tm}\leq\left(1-\frac{1}{v}\right)t\ln (tm)\left(1-\frac{3.965}{\ln^2(tm\ln (tm))}\right). $$
   On the other hand, for $n\geq2,$ (see (4.2) in \cite{3})
  $$\ln p_n\leq\ln n+\ln\ln n+1. $$
  Thus it is sufficient to choose $m$ so large that the following inequality holds
  $$\ln (tm)+\ln\ln (tm)+1\leq\left(1-\frac{1}{v}\right)t\ln (tm)\left(1-\frac{3.965}
  {\ln^2(tm\ln (tm))}\right),$$
  or, since $1-\frac{1}{v}=\frac{1}{k+1},$ that
  \begin{equation}\label{28}
   \frac {\ln (tm)+\ln\ln (tm)+1}{\ln (tm)(1-\frac{3.965}{\ln^2(tm\ln (tm))})}\leq
  \frac{t}{k+1}.
  \end{equation}
  Let, e.g., $k=1,\; t=3.$ We can choose $m_0=350.$ Then the left hand side of (\ref{28})
  equals $1.4976...<1.5\;.$ This means that that at least, for $m\geq350,$ the
  estimate (\ref{3}) and, for $m\geq351,$ the estimate (\ref{21}) are valid.
  Using a computer verification for
  $m\leq350,$ we obtain both of these estimates. Note that another short proof
  of (\ref{3}) was obtained in \cite{12} (see there Remark 32).\newline
  Other estimates of the proposition are proved in the same way.
  \end{proof}
  \section{Estimates and formulas for $N_k(m)$}
  \begin{proposition}\label{p9}
  \begin{equation}\label{29}
  N_k(1)=2,\; k=2,3,5,9,14.
  \end{equation}
  For $m\geq2,$
  \begin{equation}\label{30}
  N_k(m)\leq\left\lceil\frac{R_{\frac{k+1}{k}}(m)}{k+1}\right\rceil;
  \end{equation}
  besides, if $R_{\frac{k+1}{k}}(m)\equiv1\pmod{k+1},$ then
  \begin{equation}\label{31}
  N_k(m)=\left\lceil\frac{R_{\frac{k+1}{k}}(m)}{k+1}\right\rceil=
  \frac{R_{\frac{k+1}{k}}(m)+k}{k+1}
  \end{equation}
  and, if $R_{\frac{k+1}{k}}(m)\equiv2\pmod{k+1},$ then
   \begin{equation}\label{32}
  N_k(m)=\left\lceil\frac{R_{\frac{k+1}{k}}(m)}{k+1}\right\rceil=
  \frac{R_{\frac{k+1}{k}}(m)+k-1}{k+1}.
  \end{equation}
  \end{proposition}
  \begin{proof}
  If $m\geq2,$ formally the condition $x=(k+1)n\geq (k+1)N_k(m)$ is not stronger than the
 condition $x\geq R_{\frac{k+1}{k}}(m),$ since the first one is valid only for $x$
 multiple of $k+1.$ Therefore, for
 $m\geq2,$ (\ref{30}) holds. It allows to calculate the terms of sequence
 $\{N_k(m)\}$ for every $k>1, \;m\geq2.$ Since $N_k(1)\leq N_k(2),$ then,
  having $N_k(2),$ we also can prove (\ref{29}), using direct calculations.
   Now let $R_{\frac{k+1}{k}}(m)\equiv1\pmod{k+1}.$
  Note that, for $y=(R_{\frac{k+1}{k}}(m)-1)/(k+1)$ the interval
  \begin{equation}\label{33}
   (ky,\;(k+1)y)=\left(\frac{k}{k+1}\left(R_{\frac{k+1}{k}}(m)-1\right),\;
    R_{\frac{k+1}{k}}(m)-1\right)
    \end{equation}
  cannot contain more than $m-1$ primes. Indeed, it is an interval of type
  $\left(\frac{k}{k+1}x,x\right)$ for integer $x$ and the following such interval is
  $$\left(\frac{k}{k+1}\left(R_{\frac{k+1}{k}}(m)\right), R_{\frac{k+1}{k}}(m)\right).$$
  By the definition, $R_{\frac{k+1}{k}}(m)$ is the \emph{smallest} number
  such that if $x\geq R_{\frac{k+1}{k}}(m),$ then $\{(\frac{k}{k+1}x, x)\}$
  contains $\geq m$ primes.
   Therefore, the supposition that already interval (\ref{33})
   contains $\geq m$ primes contradicts the minimality of $R_{\frac{k+1}{k}}(m).$
    Since the following
    interval of type $(ky,\;(k+1)y)$ with integer
    $y\geq \frac{k}{k+1}(R_{\frac{k+1}{k}}(m)-1)$ is
    $$\left(\frac{k}{k+1}(R_{\frac{k+1}{k}}(m)+k),\;
    R_{\frac{k+1}{k}}(m)+k\right),$$
     then (\ref{31}) follows.

   Finally, let $R_{\frac{k+1}{k}}(m)\equiv2\pmod{k+1}.$
  Again show that, for $y=(R_{\frac{k+1}{k}}(m)-2)/(k+1)$ the interval
  \begin{equation}\label{34}
   (ky,\;(k+1)y)=\left(\frac{k}{k+1}(R_{\frac{k+1}{k}}(m)-2),\;
    R_{\frac{k+1}{k}}(m)-2\right)
 \end{equation}
 cannot contain more than $m-1$ primes. Indeed, comparing interval (\ref{34})
 with interval (\ref{33}), we see that they contain the same
 integers except for $R_{\frac{k+1}{k}}(m)-2$ which is multiple of $k+1.$ Therefore,
  they contain the same number of primes and this number does not exceed $m-1.$
  Again, since the following interval of type $(ky,\;(k+1)y)$ with integer
    $y\geq \frac{k}{k+1}(R_{\frac{k+1}{k}}(m)-2)$ is
    $$\left(\frac{k}{k+1}(R_{\frac{k+1}{k}}(m)+k-1),\;
    R_{\frac{k+1}{k}}(m)+k-1\right),$$
 then (\ref{32}) follows.
 \end{proof}
 \begin{remark}\label{r10} Obviously formulas (\ref{30})-(\ref{32}) are valid for
 not only for the considered values of $k,$ but for arbitrary $k\geq1.$
 \end{remark}
As a corollary from (\ref{29}), (\ref{31})-(\ref{32}), we obtain the following formula
 in case $k=2.$
  \begin{proposition}\label{p11}
  \begin{equation}\label{35}
  N_2(m)=\begin{cases}
2,\;\;if\;\;m=1,\;\\\left\lceil\frac{R_{\frac{3}{2}}(m)}{3}\right\rceil,
\;\;if\;\;m\geq2.\end{cases}
\end{equation}
  \end{proposition}
  Formula (\ref{35}) shows that the case $k=2$ over its regularity not concedes
  to a classic case $k=1.$
 Note that, if $k\geq3$ and $R_{\frac{k+1}{k}}(m)\equiv j\pmod{k+1},\; 3\leq j\leq k,$
   then, generally speaking, (\ref{30}) is not an equality. Evidently,
   $N_k(m)\geq N_k(m-1)$ and it is interesting that the equality is
   attainable (see below sequences (\ref{37})-(\ref{40})).
\begin{example}\label{e12} Let $k=3,\; m=2.$ Then $v=\frac{4}{3}$ and, by
  (\ref{14}), $R_{\frac{4}{3}}(2)=29\equiv1\pmod4.$ Therefore, by (\ref{31}),
   $N_3(2)=\frac{29+3}{4}=8.$ Indeed, interval $(3\cdot7,\; 4\cdot7)$ already
    contains only prime 23.
  \end{example}
  \begin{example}\label{e13} Let $k=3,\; m=3.$ Then, by
  (\ref{14}), $R_{\frac{4}{3}}(3)=59\equiv3\pmod4.$ Here $N_3(3)=11$ which
  is essentially less than $\left\lceil {R_{\frac{4}{3}}(3)/4}\right\rceil=15.$ Indeed,
   each interval
   $$(3\cdot15,\; 4\cdot15),\;(3\cdot14,\; 4\cdot14),\;(3\cdot13,\; 4\cdot13),\;(3\cdot12,\; 4\cdot12),\; (3\cdot11,\; 4\cdot11)$$
   contains more than 2 primes and only interval $(3\cdot10,\; 4\cdot10)$ contains
   only 2 primes.
  \end{example}
  In any case, Proposition \ref{p9} allows to calculate terms of sequence
  $\{N_k(m)\}$ for every considered values of $k.$ So, we obtain the following
  few terms of $\{N_k(m)\}:$\newline
  for $k=2,$
  \begin{equation}\label{36}
 2,5,13,14,23,25,33,43,46,58,60,61,71,77,80,88,103,104,...\;;
 \end{equation}
 for $k=3,$
 \begin{equation}\label{37}
 2,8,11,17,26,38,40,41,48,57,68,68,70,87,96,100,108,109,...\;;
 \end{equation}
  for $k=5,$
 \begin{equation}\label{38}
 2,7,17,24,25,38,41,58,59,64,65,73,95,97,103,106,107,108,...\;;
 \end{equation}
  for $k=9,$
\begin{equation}\label{39}
 2,14,23,23,34,36,57,58,60,60,77,86,100,100,102,123,149,...\;;
 \end{equation}
for $k=14,$
 \begin{equation}\label{40}
 2,11,24,37,38,39,50,96,96,96,96,97,97,125,125,132,178,178,...\;.
 \end{equation}
\begin{remark}\label{r14}
If, as in \cite{1}, \cite{6}, instead of intervals $(kn,\;(k+1)n),$ to
consider intervals $[kn,\;(k+1)n],$  then sequences (\ref{5}), (\ref{36})-(\ref{38})
would begin with 1.
\end{remark}
 \section{Method of small intervals}
 If we know a theorem of the type: for $x\geq x_0(\Delta),$ the interval
 $(x,\; (1+\frac{1}{\Delta})x]$ contains a prime, then we can calculate
  a bounded number of the first terms of sequences (\ref{5}) and (\ref{36})-(\ref{40}).
Indeed, put $x_1=kn,$ such that $n\geq\frac{x_0}{k}.$ Then $(k+1)n=\frac{k+1}{k}x_1$
and, if $1+\frac{1}{\Delta}<\frac{k+1}{k},$ i.e., $\Delta>k,$ then
$$\left(x_1, \;(1+\frac{1}{\Delta})x_1\right]\subset(kn, (k+1)n). $$

Thus, if $n\geq\frac{x_0}{k},$ then the interval $(kn, (k+1)n)$ contains a prime,
and, using method of finite descent, we can find $N_k(1).$ Further, put $x_2=(1+\frac{1}{\Delta})x_1.$ Then interval $(x_2,\; (1+\frac{1}{\Delta})x_2]$
also contains a prime. Thus the union
$$\left(x_1,\;(1+\frac{1}{\Delta})x_1\right]\cup\left(x_2,\;(1+\frac{1}{\Delta})x_2\right]=\left(x_1,(1+\frac{1}{\Delta})^2x_1\right]$$
contains at least two primes. This means that if
$(1+\frac{1}{\Delta})^2x_1<(k+1)n$ or $(1+\frac{1}{\Delta})^2<1+\frac{1}{k},$
then
$$\left(x_1, \;(1+\frac{1}{\Delta})^2x_1\right]\subset(kn, (k+1)n)$$
 and the interval $(kn, (k+1)n)$ contains at least two primes; again, using method of finite descent, we can find $N_k(2),$ etc., if $(1+\frac{1}{\Delta})^m<1+\frac{1}{k},$
 then
 $$\left(x_1, \;(1+\frac{1}{\Delta})^mx_1\right]\subset(kn, (k+1)n)$$
 and the interval $(kn, (k+1)n)$ contains at least $m$ primes and we can find
  $N_k(m).$ In this way, we can find $N_k(m)$ for $m<\frac{\ln(1+\frac{1}{k})}
  {\ln (1+\frac{1}{\Delta})}.$ In 2002, Ramar\'{e} and Saouter \cite{9} proved that
  interval $(x(1-28314000^{-1}),\;x)$ always contains a prime if $x>10726905041,$
  or, equivalently, interval $(x,\; (1+28313999^{-1})x)$ contains a prime if
  $x>10726905419.$ This means that, e.g., we can find $N_{14}(m)$ for
   $m \le 1954471.$ Unfortunately, this method cannot give the
   exact estimates and formulas for $N_k(m)$ as (\ref{30})-(\ref{32}).\newline
   \indent We can also to consider a more general application of of this method.
   Consider a fixed infinite set $P$ of primes which we call $P$-primes.
    Furthermore, consider the following generalization of $v$-Ramanujan numbers.
 \begin{definition}\label{d15}
  For $v>1,$ a $(v,P)$-Ramanujan number  $(R^{(P)}_v(m)),$ is the smallest
 integer such that if $x\geq R^P_v(m),$ then $\pi_P(x)-\pi_P(x/v)\geq m,$ where
 $\pi_P(x)$ is the number of $P$-primes not exceeding $x.$
 \end{definition}
 Note that every $(v,P)$-Ramanujan number is $P$-prime.
 If we know a theorem of the type: for $x\geq x_0(\Delta),$ the interval
 $\left(x,\; (1+\frac{1}{\Delta})x\right]$ contains a $P$-prime, then, using the above
 described algorithm we can calculate a bounded number of the first $(v,P)$
 -Ramanujan numbers. For example, let $P$ be the set of primes $p\equiv1\pmod{3}.$ From the
  result of Cullinan and Hajir \cite{2} it follows, in particular, that for
  $x\geq106706,$ the interval $(x,\; 1.048x)$ contains a $P$-prime. Using the
  considered algorithm, we can calculate the first 14 $(2,P)$-Ramanujan numbers.
  They are
  \begin{equation}\label{41}
 7,31,43,67,97,103,151,163,181,223,229,271,331,337.
 \end{equation}
 Analogously, if $P$ is the set of primes $p\equiv2\pmod{3},$
 then the sequence of $(2,P)$-Ramanujan numbers begins
 \begin{equation}\label{42}
 11, 23, 47, 59, 83, 107, 131, 167, 227, 233, 239, 251, 263, 281,...\;;
 \end{equation}
 if $P$ is the set of primes $p\equiv1\pmod{4},$
 then the sequence of $(2,P)$-Ramanujan numbers begins
 \begin{equation}\label{43}
 13, 37, 41, 89, 97, 109, 149, 229, 233, 241, 257, 277, 281, 317,...\;;
 \end{equation}
 and, if $P$ is the set of primes $p\equiv3\pmod{4},$
 then the sequence of $(2,P)$-Ramanujan numbers begins
 \begin{equation}\label{44}
 7, 23, 47, 67, 71, 103, 127, 167, 179, 191, 223, 227, 263, 307,...\;.
 \end{equation}
 Denote by $N^{(P)}_k(m)$ the smallest number such that, for $n\geq N^{(P)}_k(m),$
 the interval $(kn,(k+1)n)$ contains at least $m$ $P$-primes. It is easy to see that
 formulas (\ref{30})-(\ref{32}) hold for $N^{(P)}_k(m)$ and
  $R^{(P)}_{\frac{k+1}{k}}(m).$ In particular, in cases $k=1,2$ we have the formulas
   \begin{equation}\label{45}
 N^{(P)}_1(m)=\frac{R^{(P)}_2(m)+1}{2},\;\;N^{(P)}_2(m)=
 \left\lceil\frac{R^{(P)}_{\frac{3}{2}}(m)}{3}\right\rceil.
 \end{equation}
  Therefore, the following sequences for $N^{(P)}_1(m)$ for the considered cases of
  set $P$ correspond to sequences (\ref{41})-(\ref{44}) respectively:
  \begin{equation}\label{46}
 4,16,22,34,49,52,76,82,91,112,115,136,166, 169,...\;;
 \end{equation}
 \begin{equation}\label{47}
 6,12,24,30,42,54,66,84,114,117,120,126,132,141,...\;;
 \end{equation}
 \begin{equation}\label{48}
 7,19,21,45,49,55,75,115,117,121,129,139,141,159,...\;;
 \end{equation}
 \begin{equation}\label{49}
 4,12,24,34,36,52,64,84,90,96,112,114,132,154,...\;.
 \end{equation}
\section{Proof of Theorem \ref{t1}}
For $k\geq1,$ denote by $a(k)$ the least integer $n>1$ for which the interval $(kn,\;(k+1)n)$
contains no prime; in the case, when such $n$ does not exist, we put $a(k)=0.$ Taking
 into account (\ref{21}), note that $a(k)=0$ for $k=1,2,3,5,9,14,...\;.$
 Consider sequence $\{a(k)\}.$ Its first few terms are (A218831 in \cite{13})
 \begin{equation}\label{50}
 0,0,0,2,0,4,2,3,0,2,3,2,2,0,6,2,2,3,2,6,3,2,4,2,2,7,2,2,4,3,...\;.
 \end{equation}
 Calculations of $a(k)$ in the range $\{15,..., 5\times10^7\}$ lead to values of
  $a(k)$ in the interval $[2,\;16]$ which completes the proof.\;\;\;\;\;\;$\square$
\newline
\section{Acknowledgment}
The authors are grateful to  N. J. A. Sloane for several important remarks.

\end{document}